\documentclass[12pt,hyp,]{amsart}

\usepackage{rotating}
\usepackage{tikz}
\usepackage{tikz-cd}
\usepackage{subfig}
\usepackage{adjustbox}
\usepackage{cite}
\usepackage[colorlinks, linkcolor=blue, allcolors=blue]{hyperref}
\hypersetup{nesting=true,debug=true,naturalnames=true}
\usepackage[margin=1in]{geometry}

\usetikzlibrary{decorations.markings}

\title{Quotients of definite periodic knots are definite} 

\author{Keegan Boyle}  

\email{kboyle@uoregon.edu}  

\subjclass[2010]{}

\newcommand{\Z}{{\mathbb{Z}}}

\newtheorem{lemma}{Lemma}
\newtheorem{proposition}{Proposition}

\newtheorem{theorem}{Theorem}

\theoremstyle{definition}
\newtheorem{definition}{Definition}[section]

\newtheorem*{thm:main}{Theorem \ref{thm:main}}

\begin{document}
\begin{abstract}
A knot $K$ is definite if $|\sigma(K)| = 2g(K)$. We prove that the quotient of a definite periodic knot is definite by considering equivariant minimal genus Seifert surfaces.
\end{abstract}
\maketitle
\section{introduction}
Let $K$ be a knot in $S^3$ with signature $\sigma(K)$ and genus $g(K)$. Then $K$ is \emph{definite} if $|\sigma(K)| = 2g(K)$. This is a relatively small class of knots, but this condition has a nice geometric interpretation. Specifically, a knot is definite if and only if it has a Seifert surface with definite linking form. 

A knot $K \subset S^3$ is \emph{periodic} if it is fixed by a finite cyclic group acting on $S^3$ with fixed set an unknot disjoint from $K$. In this case we refer to the image of $K$ in $S^3/(\Z/p)$ as the \emph{quotient knot}. 

The goal of this paper is to investigate periodic definite knots, and in particular apply a result of Edmonds \cite[Theorem 4]{E} to prove the following theorem.
\begin{theorem}
\label{thm:main}
The quotient of a periodic definite knot is definite.
\end{theorem}
\section{background}
\begin{definition}
A quadratic form $\langle -,-\rangle$ is \emph{positive (resp. negative) definite} if $\langle x,x \rangle \geq 0$ (resp. $\leq 0$) for all $x \neq 0$.
\end{definition}
We will also use the equivalent charaterization that a matrix is positive (resp. negative) definite if and only if all of its eigenvalues are positive (resp. negative).
\begin{definition}
A Seifert surface $S$ for $K$ is \emph{positive (resp. negative) definite} if the (symmetrized) linking form lk$(-,-)$ on $H_1(S)$ as defined in \cite[Section 2]{GL} is positive (resp. negative) definite. That is, the symmetrized Seifert matrix for $S$ is definite.
\end{definition}

\begin{definition}
A knot is \emph{definite} if it has a definite Seifert surface.
\end{definition}

\begin{lemma}
\label{lem:def}
Let $K \subset S^3$ be a knot. Then the following are equivalent.
\begin{enumerate}
\item $K$ is definite.
\item Every minimal genus Seifert surface for $K$ is definite.
\item $|\sigma(K)| = 2g(K)$, where $g(K)$ is the genus of $K$.
\end{enumerate}
\end{lemma}
\begin{proof}
(2) implies (1) is obvious, and we will show that (1) implies (3) and (3) implies (2). 

To see that (1) implies (3), suppose $K$ is definite with definite Seifert surface $S$ and corresponding symmetrized Seifert matrix $M \in M_n(\Z)$. Since $M$ is definite, $\sigma(M) = \pm n = \sigma(K)$. In particular, $M$ is a minimal dimensional symmetrized Seifert matrix and so $S$ is a minimal genus Seifert surface. Hence $|\sigma(K)| = 2g(K)$.

On the other hand, suppose $|\sigma(K)| = 2g(K)$. Then taking any minimal genus Seifert surface $S$ with symmetrized Seifert matrix $M \in M_n(\Z)$, we see that $|\sigma(K)| = |\sigma(M)| \leq \dim(M) = 2g(K)$, and hence $|\sigma(M)| = n$ so $M$ is definite.
\end{proof}
The following proposition gives a strong restriction on the Alexander polynomial of definite knots.
\begin{proposition}
Let $K \subset S^3$ be a definite knot. Then $|\Delta_K(t)| = |\sigma(K)| = 2g(K)$, where $|\Delta_K(t)|$ is the width of the Alexander polynomial.
\end{proposition}
\begin{proof}
Let $S$ be a definite Seifert surface for $K$ with Seifert matrix $M \in M_n(\Z)$, and recall that $\Delta_K(t) = \det(M^T - tM)$. Multiplying both sides by $\det(M^{-1})$ makes it clear that the first and last terms of $\Delta_K(t)$ will be $\det(M)t^n$ and $\det(M)$ respectively. Since $M$ is definite, $\det(M) \neq 0$, and so the width of the Alexander polynomial is $n = |\sigma(M)| = |\sigma(K)|$. The second inequality is proved in Lemma \ref{lem:def}.
\end{proof}

\section{Periodic definite knots}
\begin{thm:main}
The quotient knot of a periodic definite knot is definite.
\end{thm:main}
The proof of this theorem relies on the following theorem of Edmonds.
\begin{theorem}\cite[Theorem 4]{E}
\label{areathm}
Let $\widetilde{K}$ be a periodic knot. Then there exists a minimal genus Seifert surface $\widetilde{S}$ for $\widetilde{K}$ which is preserved by the periodic action. Furthermore, the image of $\widetilde{S}$ in the quotient is a Seifert surface for the quotient knot $K$.
\end{theorem}

We will also need the following lemma.
\begin{lemma}
\label{lem:quot}
If the preimage of a Seifert surface $S$ under a $\Z/p$ rotation action in $S^3$ is a positive (resp. negative) definite Seifert surface $\widetilde{S}$, then $S$ is positive (resp. negative) definite. 
\end{lemma}
\begin{proof}
Consider a curve $C \subset S$ which is homologically non-trivial. Let $\widetilde{C}$ be the (possibly disconnected) preimage of $C$ in $\widetilde{S}$. Note that since $C$ is homologically non-trivial, so is $\widetilde{C}$. Now suppose $\widetilde{S}$ is positive definite so that lk$(\widetilde{C},\widetilde{C})>0$. We claim that lk$(C,C) >0$, so that $S$ is also positive definite. The linking number lk$(C,C)$ is the sum of (signed) intersection points between $C$ and the Seifert surface $\Sigma$ for a positive push-off of $C$. Let $\widetilde{\Sigma}$ be the preimage of $\Sigma$ which is an equivariant Seifert surface for a positive push-off of $\widetilde{C}$. Then each intersection point between $C$ and $\Sigma$ lifts to $p$ intersection points (with the same sign) between $\widetilde{C}$ and $\widetilde{\Sigma}$. Hence lk$(\widetilde{C},\widetilde{C}) = p \cdot$lk$(C,C)$, and so lk$(C,C) > 0$.
\end{proof}

\begin{proof}[Proof of Theorem \ref{thm:main}]
By Theorem \ref{areathm} any periodic knot $\widetilde{K}$ has an equivariant minimal genus Seifert surface $\widetilde{S}$ with quotient $S$. By Lemma \ref{lem:def}, $\widetilde{S}$ is definite, and so by Lemma \ref{lem:quot} $S$ is as well.
\end{proof}

\bibliography{bibliography}{}

\begin{thebibliography}{Edm84}

\bibitem[Edm84]{E}
Allan~L. Edmonds.
\newblock Least area {S}eifert surfaces and periodic knots.
\newblock {\em Topology Appl.}, 18(2-3):109--113, 1984.

\bibitem[GL78]{GL}
C.~McA. Gordon and R.~A. Litherland.
\newblock On the signature of a link.
\newblock {\em Invent. Math.}, 47(1):53--69, 1978.

\end{thebibliography}
\bibliographystyle{alpha}

\end{document}